\documentclass[11pt]{amsart}
\usepackage[a4paper,margin=3cm]{geometry}
\usepackage{amsmath}
\usepackage{amssymb}

\newcommand{\R}{\mathbb R}
\newcommand{\dd}[1]{\frac{\partial}{\partial #1}}
\newcommand{\E}{\mathcal{E}}
\newcommand{\JT}[1]{J^{#1}(\R,\R)}

\newcommand{\sym}{\operatorname{sym}}

\newcommand{\Ker}{\operatorname{Ker}}
\newcommand{\Fol}{\operatorname{Fol}}

\newtheorem{lem}{Lemma}
\newtheorem{thm}{Theorem}
\newtheorem{cor}{Corollary}
\newtheorem{prop}{Proposition}
\newtheorem{defin}{Definition}
\theoremstyle{remark}
\newtheorem{rem}{Remark}

\begin{document}

\title{Equivalence of variational problems of higher order}
\author{Boris Doubrov \and Igor Zelenko}
\address{Boris Doubrov, Department of Applied Mathematics and Computer Science, Belarussian State University, Nezavisimosti Ave.~4, Minsk 220030, Belarus}
\email{doubrov@islc.org}
\address{Igor Zelenko, Department of Mathematics, Texas A$\&$M University, College Station, TX 77843-3368, USA}
\email{zelenko@math.tamu.edu}

\begin{abstract}
We show that for $n\ge 3$ the following equivalence problems are essentially the same:  the equivalence problem for Lagrangians of order $n$ with one dependent and one independent
variable considered up to a contact transformation, a multiplication by a nonzero constant, and modulo divergence; the equivalence problem for the special class of rank 2 distributions associated with underdetermined ODEs $z'=f(x,y,y', \dots, y^{(n)})$; the equivalence problem for variational ODEs of order $2n$. This leads to new results such as the fundamental system of invariants for all these problems and the explicit description of the maximally symmetric models. The central role in all three equivalence problems is played by the geometry of self-dual curves in the projective space of odd dimension up to projective transformations via the linearization procedure (along the solutions of ODE or abnormal extremals of distributions). More precisely, we show that an object from one of the three equivalence problem
is maximally symmetric if and only if all curves in projective spaces obtained by the linearization procedure
are rational normal curves.

\end{abstract}
\maketitle

\section{Introduction: three equivalence problems}
\setcounter{equation}{0}
\setcounter{thm}{0}
\setcounter{lem}{0}
\setcounter{prop}{0}
\setcounter{cor}{0}
\setcounter{rem}{0}

The main goal of this paper is to establish a tight relationship between the following three
local equivalence problems in differential geometry:
\begin{enumerate}
\item equivalence of variational problems of order $\ge3$ with one dependent and one independent
variable considered up to a contact transformation, a multiplication by a constant,  and modulo divergence;
\item equivalence of \emph{variational ODEs} of order $\ge 6$ up to contact transformations. By a variational ODE (called also variational with multiplier) we mean  an ODE which is contact equivalent to  an Euler-Lagrange equations for some Lagrangian.
\item equivalence of rank 2 distributions 
associated with underdetermined ODEs $z'=f(x,y,y', \dots, y^{(n)})$, $n\ge 3$.
\end{enumerate}
In particular, we shall show that equivalence problems (1), (2), and the equivalence problem for the particular class of rank 2 distributions mentioned in item (3) are essentially the same. In particular, there is a one-to-one
correspondence between equivalence classes of objects in all these problems.

The one-to-one correspondence (up to above equivalence relation) between Lagrangians of order $n\ge 2$ and their Euler--Lagrange equations was already established earlier in works of M.~Fels~\cite{fels} for $n=2$ and M.~Jur\'a\v s~\cite{juras} for $n\ge 3$. It is based on the characterization of variational ODEs in terms of the variational bicomplex given in~\cite{and}. We extend this correspondence to underdetermined ODEs and
the corresponding rank 2 vector distributions in the case $n\geq 3$. This allows us to apply the results of our previous works \cite{dz1,dz2}, where more general rank 2 distributions are treated, for the description of the unique maximally symmetric Lagrangian up to the considered equivalence relation (see also the discussions on various equivalence relations for variational problems at the end of subsection~\ref{varintro}). Note that the one-to-one correspondence between the equivalence problems (1) and (3) does not hold for $n=2$. 
For example, the
Lagrangian $(y'')^{1/3}\,dx$ is not equivalent to the most symmetric one $(y'')^2\,dx$, but the corresponding
underdetermined ODEs and rank 2 distributions are equivalent and have 14-dimensional Lie algebra $G_2$ as their symmetry.

The common feature of all three problems above is that they admit linearization, which reduces them in essence to the problem of equivalence of self-dual curves in odd-dimensional projective spaces up
to projective transformations. The invariants
of these curves in projective spaces, the Wilczynski invariants~\cite{wil}, produce the invariants of the original problem. The latter are called the generalized Wilczynski invariants.

In this work we exploit an alternative (a Hamiltonian) point of view on the variational problems, which comes from the Pontryagin Maximum Principle in Optimal Control. This point of view provides us with the Hamiltonian form of the Euler-Lagrange equation
and allows to construct a (generalized) Legendre transform that takes extremals of the
Lagrangian (or, in other words, the solutions of the corresponding Euler--Lagrange equation) to the
abnormal extremals of the corresponding optimal control problem. This immediately shows that the solution
space of any variational ODE carries a natural symplectic structure $\omega$. We show that the conformal class of this symplectic structure (i.e., all 2-forms $f\omega$ for non-vanishing functions $f$) can be recovered
only from the self-duality of all linearizations of the given ODE along its solutions. This, in its turn, can
be reformulated by vanishing of the generalized Wilczynski invariants (see~\cite{wil}) of odd order. It is easy to see that for $n\ge 2$ there is at most one (up to constant) closed 2-form in any given conformal class of non-degenerate 2-forms on a smooth manifold. This gives a `naive' prove that any variational ODE of order $\ge 4$ admits at most one Lagrangian up to a constant and divergence terms.

Another question we try to answer in this paper is whether invariants of the above three equivalence problems derived from the Wilczynski invariants of self-dual projective curves, provide the complete system of fundamental invariants. It has been known from~\cite{dou,douwil} that the answer to the similar question for arbitrary non-linear ODEs is negative. Namely, there exist non-trivial ODEs (i.e., equations, not equivalent to $y^{(n)}=0$ via contact transformations) such that all their linearizations are trivial and, thus, all their Wilczynski invariants vanish. However, surprisingly, for variational ODEs the answer is positive and Wilczynski invariants of even order provide the complete system of fundamental invariants for this class of equations.

We show that for variational ODEs the generalized Wilczynski invariants of even order (Wilczynski invariants of odd order automatically vanish due to the self-duality of the linearization) form a fundamental set of
contact invariants in the following sence. Any other differential invariant lies in a radical of a differential
ideal generated by these invariants. In particular, vanishing of this fundamental set of invariants implies that any other differential invariant of the variantional equation vanishes identically and the equation is contact equivalent to the tivial one. We note that for $n=3$ and $4$ these invariants do not generate the complete differential algebra of invariants by only differentiation and algebraic operations. In particular, for $n=3$ (or 6-th order variational ODEs $y^{(6)}=F(x,y,y',\dots,y^{(5)})$) the differential invariant $I=F_{45}$ satisfies a non-trivial qubic equation, whose coefficients are certain derivatives of the generalized Wilczynski invariant of order~$4$.

To summarize vanishing of generalized Wilczynski invariants of even order gives an explicit characterization
of the most symmetric models in all three equivalence problems provided that $n\ge 3$:
\begin{enumerate}
\item all variational ODEs with vanishing generalized Wilczynski invariants are contact equivalent to the trivial equation $y^{(2n)}=0$;
\item all Lagrangians with vanishing generalized Wilczynski invariants are equivalent to $(y^{(n)})^2\,dx$ modulo constant multiplier, contact transformations and divergence terms;
\item all underdetermined ODEs with vanishing generalized Wilczynski invariants equivalent to $z'=(y^{(n)})^2$.
\end{enumerate}

Let us we briefly outline each of equivalence problems (1)-(3).

\subsection{Equivalence of variational problems}
\label{varintro}
This paper deals with variational problems in one dependent and one independent
variable of arbitrary order. A variational problem is defined by a Lagrangian
$L=f(x,y,y',\dots,y^{(n)})\,dx$ or the corresponding functional $\int
f(x,y,y',\dots,y^{(n)})\,dx$.

Let us recall basic definitions from the geometry of variational problems. Let
$J^\infty=J^\infty(\R,\R)$ be an infinite jets of smooth functions $y(x)$. We
shall use the standard coordinate system $(x,y=y_0, y_1, y_2, \dots)$ on
$J^\infty$. Denote by $\theta_i= dy_i - y_{i+1}\,dx$, $i\ge 0$, the basis of
so-called \emph{contact forms} on $J^\infty$. The set of all exterior forms
$\Lambda(J^\infty)$ is naturally turned into the bi-graded algebra with:
\begin{align*}
\Lambda^{p}(J^\infty) &= \Lambda^{0,p} \oplus
\Lambda^{1,p-1},\quad\text{where}\\
\Lambda^{0,p} &= \langle \theta_{i_1} \wedge \dots \wedge \theta_{i_p}\rangle, p\ge 0; \\
\Lambda^{1,p-1} &= \langle \theta_{i_1} \wedge \dots \wedge
\theta_{i_{p-1}}\wedge dx\rangle, p\ge 1.
\end{align*}
The exterior derivative $d\colon \Lambda^{p}(J^\infty)\to
\Lambda^{p+1}(J^\infty)$ naturally splits into the sum $d=d_H+d_V$, where for
$\omega\in \Lambda^{p}(J^\infty)$ we have $d_H(\omega)\in \Lambda^{1,p}$ and
$d_V(\omega)\in \Lambda^{0,p+1}$.

We consider variational problems $\int f(x,y,y',\dots,y^{(n)})\,dx$ of
arbitrary order~$n$ up to the divergence equivalence and constant multiplier.
Namely, we say that two Lagrangians $L_1 = f_1(x,y,y',\dots,y^{(n)})\,dx$ and
$L_2= f_2(x,y,y',\dots,y^{(n)})\,dx$ are equivalent if there exists a contact
transformation $\phi\colon \JT{1}\to \JT{1}$ with the prolongation $\Phi\colon
\JT{\infty}\to \JT{\infty}$, such that
\begin{equation}\label{Lequiv}
\Phi^*(L_2)=\alpha L_1 + d_H(\mu) \mod \langle \theta_i \mid i\ge 0\rangle
\end{equation}
for some constant nozero $\alpha\in\R$ and function $\mu$ on $\JT{\infty}$. We shall
always assume that all our Lagrangians are non-degenerate, i.e., they are
non-linear in the highest derivative. It follows from~\cite{juras} that two
Lagrangians are equivalent under the above equivalence relation if and only if
their Euler--Lagrange equations are contact equivalent.

The variational equivalence problem was treated in a number of papers using
both naive approach and Cartan's equivalence
method~\cite{bryant,ko1,ko2,ko3,moor,gonz}. See also~\cite{olver} for the
symmetry classification of higher order Lagrangians.

We note that usually slightly different equivalence notion of divergence
equivalence is considered, where the constant $\alpha$ above is equal to~$1$
identically. The upper bound for the variational symmetry algebra in case of
$n$-order Lagrangian was proved to be equal to $2n+3$ for $n\ge2$ in the work
of Gonzalez-Lopez~\cite{gonz}. Note that this upper bound is achieved in
case of a family of non-equivalent Lagrangians, which is different from most of
the classical local equivalence problems in differential geometry. We show in
this paper, that adding this constant $\alpha$ in the definition~\eqref{Lequiv}
of the divergence equivalence changes this patten. In this case we get a
slightly higher upper bound equal to $2n+5$ with a unique maximally symmetric
Lagrangian equivalent to $L=(y^{(n)})^2\,dx$.

\subsection{Equivalence of rank 2 vector distributions}
By a rank 2 vector distribution on a smooth manifold $M$ we understand a
two-dimensional subbundle $D$ of the tangent vector bundle $TM$. We define its
(small) derived flag $\{D^i\}$ as follows:
\begin{align*}
D^1&=D,\\
D^{i+1} &= D^i + [D,D^i],\quad i\ge 1,
\end{align*}
and assume that the distribution $D$ is regular in a sense that all $D^i$ are
smooth subbundles of the tangent bundle $TM$. We shall also assume in the
sequel that the distribution $D$ is \emph{completely non-holonomic},
i.e.~$D^n=TM$ for some sufficiently large $n$.

We say that two such distributions $D$ and $D'$ on manifolds $M$ and $M'$ are
(locally) equivalent if there exists a (local) diffeomorphism $\phi\colon M\to
M'$, such that $\phi_*(D)=D'$.

Equivalence problem for non-holonomic distributions is an old problem, which
goes back to the end of 19th century and was studied by various mathematicians
including Lie, Goursat, Darboux, Engel, Elie Cartan and others. Except for
several cases such as rank 2 distributions on 3- and 4-dimensional manifolds,
generic rank 2 distributions have functional, and, thus, non-trivial
differential invariants. In his classical paper~\cite{cartan10} Elie Cartan
associates a $(2,5)$-distribution to a system of partial differential equations
of second order and constructs a canonical coframe for non-degenerate
distributions of this type. This is a first example of an explicit solution for
the equivalence problem of vector distributions with non-trivial functional
invariants. Remarkably, the most symmetric $(2,5)$-distributions form one
equivalence class and have an exceptional Lie algebra $G_2$ as their symmetry
algebra.

The obvious (but very rough in the most cases) discrete invariant of a
distribution $D$ at a point $q$ is the so-called \emph{small growth vector}
(s.g.v.) at $q$. It is the tuple $\{\dim D^j(q)\}_{j\in{\mathbb N}}$, where
$D^j$. Furthermore, at each point $q\in M$, where $\dim D^j$ are locally
constant for any $j$, we can consider the graded space
$\mathfrak{m}_q=\sum_{i\ge 1} D^{j+1}(q)/D^j(q)$. It can be naturally equipped
with a structure of a graded nilpotent Lie algebra and it is called a
\emph{symbol} of the distribution $D$ at a point $q$. The notion of symbol is
extensively used in works of N.~Tanaka and his school
(\cite{tan1,tan2,mori,yamag}) who systematized and generalized the Cartan
equivalence method.

However, all constructions of Tanaka theory strongly depend on the algebraic
structure of the symbol and they were carried out under the very restrictive
assumption that symbol algebras are isomorphic at different points. An
alternative approach for studying rank 2 distributions was presented by the
authors in~\cite{dz1,dz2}. It is based on the ideas of the geometric control
theory and uses a symplectification of the problem by lifting the
distribution~$D$ to the cotangent bundle. This provides an
effective way to construct a canonical coframe and, thus, solve equivalence
problem for rank~2 distributions of so-called \emph{maximal class} (this notion is defined in section \ref{lindist}).

Rank 2 distributions of a special type are naturally associated with Lagrangians.
Namely, to a variational problem with a Lagrangian $f(x,y,y',\dots,y^{(n)})\,dx$ one can assign the following (affine) control system:
\begin{equation}\label{control}
\begin{aligned}
\dot x(t)&=1\\
\dot y_i(t)&=y_{i+1}(t),\quad 0\leq i\leq n-1\\
\dot y_n(t)&=u(t)\\
\dot z(t)&=f(x(t),y_0(t),y_1(t),\dots,y_n(t)),
\end{aligned}
\end{equation}
on $\JT{n}\times \mathbb R $ with coordinates $(x,y_0,\ldots y_n,z)$, where $u(\cdot)$ is a control function belonging say to the space $L_\infty$. To any point $q_0\in \JT{n}\times \mathbb R$ and a control function $u(\cdot)$ consider the trajectory of the system \eqref{control} started at $q_0$. Such trajectory is called an \emph {admissible trajectory of control system \eqref{control} } and its velocity at $q_0$
is called \emph{an admissible velocity of control system \eqref{control} at $q_0$}.
Taking the linear span of all admissible velocities of \eqref{control} we get
the rank 2 distribution on $\JT{n}\times \mathbb R$ generated by the following two vector fields:
\begin{equation}
\label{basis}
\begin{aligned}
X_1 &= \dd{x} + y_1\dd{y_0} + \dots + y_n\dd{y_{n-1}} + f\dd{z},\\
X_2 &= \dd{y_n}.
\end{aligned}
\end{equation}
We say that this rank 2 distribution is \emph{associated with the Lagrangian $f(x,y,y',\dots,y^{(n)})\,dx$ or with underdetermined
differential equation}
\begin{equation}\label{eq1}
z'=f(x,y,y',\dots,y^{(n)})
\end{equation}
 on two unknown functions $y(x)$ and $z(x)$. Such underdetermined ODEs and the related
geometric structures have been extensively studied in~\cite{hilbert,goursat,chras2,nurow,andkru}.

It is easy to see that if $\frac{\partial^2 f}{\partial y_n^2}\ne0$ (i.e. the Lagrangian satisfies the Legendre condition), then $\dim
D^2 = 3$, $\dim D^i=i+2$ for $i=3,\dots,n+1$. The case, when $f$ is linear with
respect to $y_n$ is special, since in this case we can reduce the corresponding underdetermined
equation to the equation of the same type, but of lower order. So, in the
sequel we shall always assume that function $f$ in the right hand of the
equation~\eqref{eq1} is always non-linear with respect to $y^{(n)}$. Note that such rank 2 distributions are of maximal class.
Finally we cite the main result of \cite{dz1,dz2} that will be needed in the sequel:

\begin{thm}
\label{thmdz}
For any $(2,n+3)$-distribution, $n>2$, of maximal class
there exists a canonical frame on a
$(2n+5)$-dimensional bundle over $M$.
The group of symmetries of such
distribution is at most $(2n+5)$-dimensional. Any
$(2,n+3)$-distribution of maximal class with
$(2n+5)$-dimensional group of symmetries is locally
equivalent to the distribution, associated with the
Lagrangian $\bigl(y^{(n)}(x)\bigr)^2$. The
algebra of infinitesimal symmetries of this distribution is
isomorphic to a semidirect sum of $\mathfrak{gl}(2,\mathbb
R)$ and $(2n+1)$-dimensional Heisenberg algebra ${\mathfrak
n}_{2n+1}$.
\end{thm}

\subsection{Equivalence of ordinary differential equations}
\label{ODEsubsec}
We shall also consider the equivalence problem of scalar ordinary differential
equations of the form
\begin{equation}
\label{ODE}
y^{(N+1)} = F(x,y,y',\dots,y^{(N)}).
\end{equation}
Each such equation can be considered as a hypersurface $\E$ in the jet space
$\JT{N+1}$. We shall always assume that our equations are solved with respect
to the highest derivative, so that the restriction of the natural projection
$\pi_{N+1,N}\colon \JT{N+1}\to\JT{N}$ to the hypersurface $\E$ is a diffeomorphism.

Two such equations $\E$ and $\E'$ are said to be \emph{contact equivalent}, if
there exists a contact transformation $\Phi\colon \JT{1}\to\JT{1}$ with a
prolongation $\Phi^{N+1}$ to $\JT{N+1}$ such that $\Phi^{N+1}(\E)=\E'$.

The equivalence problem of ordinary differential equations under contact and
point transformations is yet another classical subject going back to the works of
Lie, Tresse, Elie Cartan~\cite{cartan24}, Chern~\cite{chern50},
M.~Fels~\cite{fels} and others. The complete solution for the equivalence
problem was obtained in~\cite{dkm} based on the Tanaka theory of geometric
structures on filtered manifolds~\cite{tan1,tan2,mori}.

Explicit formulas for the basis in the differential algebra of contact
invariants of a single ODE of arbitrary order were computed by
B.~Doubrov~\cite{dou}. In particular, a part of these invariants comes from the
linearization of the given ODE. In fact, they coincide with classical
Wilczynski invariants of linear differential equations formally applied to the
linearization of a non-linear ODE. See~\cite{douwil} for more details.

In this paper we are mainly interested in a special class of ordinary
differential equations, consisting of equation which are contact equivalent to Euler--Lagrange equations of
variational problems. Such equations are usually called variational (with multiplier)
and have been studied in many papers~\cite{and,juras,fels}.

From the general result of Anderson and Thompson~\cite[Theorem 2.6]{and} it is known that
a scalar ordinary differential equation of order $2n$ is variational, if and only if there exists a
closed 2-form:
\begin{equation}\label{omega_and}
\omega = \sum_{i=0}^{n-1}\sum_{j=i+1}^{2n-i-1}A_{i,j}\theta_i\wedge\theta_j
\end{equation}
where $A_{n-1,n}\ne0$ and $\theta_i=dy_i - y_{i+1}dx$ are contact forms on the jet space $\JT{2n}$ restricted to the equation $\E = \{ y_{2n} = F(x,y_0,\dots,y_{2n-1}) \}$.

\subsection{Equivalence of curves in projective spaces}
Surprisingly, the central role in all three equivalence problems is played by the geometry
of self-dual curves in the projective space of odd dimension up to projective transformations.

Let $\gamma\subset P^{N}$ be an arbitrary curve in the projective space. We shall always assume that
$\gamma$ is strongly regular, i.e. its flag of osculating spaces does not have any singularities. In particular, $\gamma$ itself does not lie in any proper linear subspace of the projective space. We shall not assume any
distinguished parameter on $\gamma$, though there is always a distinguished family of
so-called projective parameters on $\gamma$.

Let $t$ be an arbitrary parameter on $\gamma$ and let $e_0(t)$ be such curve in $\R^{N+1}$ that $\gamma(t) = \R e_0(t)$. Define $e_i(t) = e_0^{(i)}(t)$. Then \emph{$i$-th osculating space $\gamma^{(i)}(t)$} of the curve $gamma$ at a point $\gamma(t)$ is defined as:
\begin{equation}
\label{osc}
\gamma^{(i)}(t) = \langle e_0(t), \dots, e_i(t) \rangle,\quad i=0,\dots, N.
\end{equation}
It is easy to see that it does not depend on the choice of the parameter $t$ and the curve $e_0(t)$. The
$(N-1)$-st osculating spaces $\gamma^{(N-1)}$ define the curve in the dual projective space $P^{N,*}$, which is
called a dual curve and is denoted by $\gamma^*$. We shall call a curve $\gamma$ \emph{self-dual}, if there
exists a projective mapping $P^{N}\to P^{N,*}$ that maps $\gamma$ to $\gamma^*$ so that any point $x\in\gamma$ is mapped to the point in  $P^{N,*}$ annihilating the $(N-1)$-st osculating space $\gamma^{(N-1)}$ to $\gamma$ at $x$ .
We summarize the properties of self-dual curves in the following proposition. For the proofs we refer to the classical book of Wilczynski~\cite{wil}):

\begin{prop}
\label{selfdual}
\begin{enumerate}
\item
If the mapping  $P^{N}\to P^{N,*}$ sending $\gamma$ to $\gamma^*$ exists, then it is unique, up to a constant nonzero factor. It defines, a unique, up to a constant nonzero factor, non-degenerate
bilinear form $\beta$ on the vector space $\R^{N+1}$.  Moreover, this form is necessarily
skew-symmetric if $N$ is odd and symmetric, if $N$ is even.
\item
In case of odd  $N$ the curve $\gamma$ is  self-dual if and only if there exists a non-degenerate skew-symmmetric (i.e. symplectic) form on $\mathbb R^{N+1}$ such that all osculating spaces $\gamma^{(N+1)/2}(t)$ are Lagrangian with respect to this form.
\end{enumerate}
\end{prop}

The invariants of projective curves (up to projective transformations) were also described by Wilczynski~\cite{wil}. The algebra of all invariants admits a basis of so-called fundamental invariants $W_3, \dots, W_{N+1}$ of order $3,\dots,N+1$ respectively. They can be constructed as follows. As above, let
$e_0(t)$ be a curve in $\R^{N+1}$ such that $\gamma = \R e_0(t)$. If the curve $\gamma$ is strongly linear, then the vectors $e_i(t) = e_0^{(i)}(t)$, $i=0,\dots,N$ form a so-called moving frame along $\gamma$. Then the next derivative $e_N'(t)=e_0^{(N+1)}(t)$ can be uniquely expressed as a linear combination of vectors in this frame. In other words, we have a well-defined linear homogeneous differential equation on $e_0(t)$:
\begin{equation}\label{def_eqn}
e_0^{(N+1)} = p_N(t) e_0^{(N)} + \dots + p_0(t) e_0.
\end{equation}
Since $e_0(t)$ is defined up to a scale, we can always fix this scaling factor by the condition $p_N(t)=0$. It is easy to see that this defines $e_0(t)$ uniquely up to a contact non-zero scale. Next, by reparametrizing the curve $e_t(0)$, i.e., by changing $t$ to $\bar t = \lambda(t)$ we can also achieve $p_N(t)=p_{N-1}(t)=0$. This fixes a parameter $t$ up to projective reparametrizations $\bar t = (at+b)/(ct+d)$. Wilczynski proves that taking linear combinations of the derivatives of the remaining coefficients $p_i(t)$, $i=0,\dots,N-2$ we can form $(N-2)$ (relative) invariants of the curve $\gamma$ under the group of projective transformations:
\[
W_k = \sum_{j=1} ^{k-2} (-1)^{j+1}\frac{(2k-j-1)!(N-k+j)!}{(k-j)!(j-1)!}p^{(j-1)}_{N-k+j},\quad k=3,\dots,N+1.
\]
Following Wilczynski, we shall say that an invariant $W_i$ has order $i$, $i=3,\dots,N+1$. Wilczynski proves that any other projective invariant of $\gamma$ can be expressed as a function of invariants $W_i$ and their derivatives. He also shows that in case of odd $N$ the curve $\gamma$ is self-dual if and only if all invariants of \emph{odd} order vanish identically. Note also that all  Wilczynski invariants of a curve $\gamma\subset P^N$ vanish if and only if $\gamma$ is a rational normal curve, i.e. it can be represented as $t\mapsto [1:t:\ldots:t^N]$ in some homogeneous coordinates of $P^N$.

Self-dual curves $\gamma$ in odd-dimensional projective spaces appear naturally in the above
equivalence problems via the linearization procedure (see section \ref{linsect} for more detail).
The linearization of ODE along  a solution assigns  a curve in projective space to the solution via identification of linear equations with curves in projective space. If the ODE is variational, then the corresponding curves in projective space are self-dual. In the case of rank 2 distributions it is not immediately clear what is the analog of solutions and what is the linearization procedure. This becomes clear if one considers distributions as the constraints for a  variational problem and use the Pontryagin Maximum Principle: the analogs of solutions of ODE's are so-called abnormal extremals of the distribution and the linearization of the flow of abnormal extremal leads to the notion of Jacobi curves introduced in~\cite{zel06,dz1,dz2}, which essentially are (or generated by) self-dual curves in a projective space. In particular, the invariants of these curves define the invariants of the original equivalence problems. For example, as shown in~\cite{zel06c}, the fundamental invariant $W_4$ of self-dual curves in $\mathbb{RP}^3$ can be identified with the so-called fundamental tensor of rank 2 vector distributions in 5 dimensional spaces discovered by E.~Cartan~\cite{cartan10}.

\section{Variational problems and rank 2 vector distributions}
\setcounter{equation}{0}
\setcounter{thm}{0}
\setcounter{lem}{0}
\setcounter{prop}{0}
\setcounter{cor}{0}
\setcounter{rem}{0}

The aim of this section is to establish the correspondence between variational
problems of order $n\ge3$ and special rank 2 vector distributions associated
with the underdetermined ordinary differential equations~\eqref{eq1} of order $n$.

\begin{lem}\label{lem1}
Let $D$ be the rank $2$ distribution associated with to the underdetermined
$z'=f(x,y,y',\dots,y^{(n)})$  of order $n\ge3$. Then the space of all infinitesimal symmetries of $D$ lying in $D^3$
is one-dimensional (over $\R$) and is
generated by the vector field $Z=\dd{z}$.
\end{lem}
\begin{proof}
Let us show that $\sym(D)\cap D^3$ is one-dimensional (over $\R$) and is
generated by the vector field $\dd{z}$. Indeed, it is easy to check that the
space $D^3$ is generated by the vector fields:
\begin{align*}
X_1 &= \dd{x} + y_1\dd{y_0} + \dots + y_n\dd{y_{n-1}} + f\dd{z},\\
X_2 &= \dd{y_n},\quad X_3 = \dd{y_{n-1}},\quad X_4 = \dd{y_{n-2}},\quad X_5 = \dd{z}.
\end{align*}
Let $Y=\sum_{i=1}^5 a_iX_i$ be an arbitrary vector field lying in $D^3$. Then
we have
\begin{multline*}
[X_1, Y] = [X_1, \sum_{i=2}^5 a_iX_i] \mod D = \\
\big(X_1(a_3)-a_2\big)\dd{y_{n-1}} +
\big(X_1(a_4)-a_3)\dd{y_{n-2}}-a_4\dd{y_{n-3}}+X_1(a_5)\dd{z} \mod D.
\end{multline*}
Thus, the condition $Y\in \sym(D)$ implies that $a_2=a_3=a_4=0$ and
$X_1(a_5)=0$. Further, we have
\[
[X_2, Y] = X_2(a_5)\dd{z} - a_1\dd{y_{n-2}} \mod D.
\]
Again, the condition $Y\in \sym(D)$ implies that $a_1=0$ and $X_2(a_5)=0$. In
particular, we see that the function $a_5$ is a first integral of the
distribution $D$. But since $D$ is completely non-holonomic, $a_5$ should be a
constant. This completes the proof of the lemma.
\end{proof}
\begin{cor}\label{lem1cor}
Let $z'=f_i(x,y,y',\dots,y^{(n)})$, $i=1,2$, be two underdetermined
differential equations of order $n\ge3$, and let $D_i$, $i=1,2$, be the
corresponding rank 2 vector distributions on $\R^{n+3}$. Suppose that
distributions $D_1$ and $D_2$ are locally equivalent. Then the equivalence
mapping $\phi$ maps vector field $\dd{z}$ to $c\dd{z}$ for some constant
$c\in\R^*$.
\end{cor}
Let us identify the space $\R^{n+3}$ with the direct product of $\JT{n}$ with
the coordinates $(x,y_0,\dots,y_n)$ and $\R$ with the coordinate $z$ and
consider any equivalence mapping $\phi$ as a mapping from $\JT{n}\times\R$ to
itself. Then Lemma~\ref{lem1} implies that any such mapping $\phi$ has the form
\begin{equation}\label{eq2}
\phi\colon \JT{n}\times\R \to \JT{n}\times\R,\quad (p,z)\mapsto (\psi(p),
\alpha z + \mu(p)),
\end{equation}
where $\psi\colon \JT{n}\to\JT{n}$, $\alpha\in\R^*$ and $\mu$ is a smooth
function on $\JT{n}$.

\begin{lem}\label{l3}
Mapping $\psi$ is a contact transformation and the function $\mu$ does not
depend on $y_n$, i.e., it is a pull-back of the function on $\JT{n-1}$.
\end{lem}
\begin{proof}
Since $\dd{z}$ is a symmetry of both distributions $D_1$ and $D_2$, we
can consider the direct images of these distributions with respect to
the natural projection $\JT{n}\times\R\to \JT{n}$. It is easy to see
that in both cases these images coinside with the contact distribution on
$\JT{n}$. This proves that $\psi$ is a contact transformation.

The second statement of the lemma on the function $\mu$ follows
immediately from the fact that both $D_1$ and $D_2$ contain the vector
field $\dd{y_n}$.
\end{proof}

\begin{thm}\label{thm_l2d} Suppose that $n\ge3$. Two vector distributions $D_1$ and $D_2$
associated
with Lagrangians
$L_i=f_i(x,y,y',\dots,y^{(n)}) dx$, $i=1,2$,
for $n\ge3$ are equivalent if and only if the Lagrangians $L_1$ and $L_2$
are equivalent.
\end{thm}
\begin{proof}
It is easy to see that the transformations~\eqref{eq2} with mappings $\psi$ and
$\mu$ satisfying Lemma~\ref{l3} induce the same equivalence relation on
distributions $D_1$ and $D_2$ as the equivalence relation on Lagrangians $L_1$
and $L_2$ given by equation~\eqref{Lequiv}.
\end{proof}

As a direct consequence of Theorems \ref{thmdz} and \ref{thm_l2d}

\begin{thm}
\label{maxLagr}
The dimension of the group of variational symmetries of Lagrangian of order $n\geq 3$ does not exceed $2n+5$.
The Lagrangian with $(2n+5)$-dimensional group of variational symmetries is equivalent to the Lagrangian $(y^{(n)})^2$.
The algebra of infinitesimal symmetries of the latter  Lagrangian is isomorphic to a semidirect sum of $\mathfrak{gl}(2,\mathbb R)$ and $(2n+1)$-dimensional Heisenberg algebra ${\mathfrak n}_{2n+1}$.
\end{thm}

 From the proof of Lemma \ref{l3} it follows that rank 2 distribution associated with some Lagrangian $f(x,y,y',\dots,y^{(n)})\,dx$ with $\frac{\partial^2 f}{\partial y_n^2}\ne0$  can be described in the following coordinate free way:
\begin{prop}
A rank 2 distribution $D$ is associated with a Lagrangian $f(x,y,\dots,y^{(n)})\,dx$ with $\frac{\partial^2 f}{\partial y_n^2}\ne0$ in a neighborhood of a generic point if and only if
\begin{enumerate}
\item $\dim D^3=5$;
\item There exists an infinitesimal symmetry $X$ of $D$ lying in $D^3$
such that the factorization by the foliation of integral curves of $X$ sends $D$ to the Goursat distribution on the quotient manifold.
\end{enumerate}
\end{prop}



\section{Two points of view on extremals of variational problems}
\setcounter{equation}{0}
\setcounter{thm}{0}
\setcounter{lem}{0}
\setcounter{prop}{0}
\setcounter{cor}{0}
\setcounter{rem}{0}

In this section we introduce abnormal extremals of rank 2 distributions and show how the flow of abnormal extremals of a distribution associated with a Lagrangian $L=f(x,y,y',\dots,y^{(n)})\,dx$ can be related to the flow of extremals of the corresponding variational problems. Speaking informally,  this relation is the relation between the Hamiltonian and the Lagrangian approach to variational problems and it is given by a kind of Legendre transform. The material of this section is pretty standard but, as we shall see in the next sections, it is very useful for the equivalence problem for Lagrangians and to our knowledge it was never used before in this kind of problems.

\subsection{Hamiltonian form of Euler-Lagrange equation}
\label{leg}
Recall that extremals of the Lagrangian $L$ are critical points of the corresponding
functional $L=\int f(x,y,y',\dots,y^{(n)})\,dx$.
On one hand, they are solutions of the Euler --Lagrange equation
\begin{equation}
\label{EL}
f_{y_0}-\frac{d}{dx}(f_{y_1}) + \dots + (-1)^n\frac{d^n}{d x^n}(f_{y_n}) = 0.
\end{equation}

If one takes a little bit more general point of view (that is standard in the Optimal Control Theory), then
the extremals can be also described using the notion of the end-point mappings associated with the corresponding control system
\eqref{control}.  Fix a point $q_0\in \JT{n}\times \mathbb R$ and a time $T>0$. \emph{The endpoint map $\mathcal F_{q_0, T}$} is the map from the space $L_\infty[0,T]$ to $\JT{n}\times \mathbb R$ sending a control function $u(t)$ to the point of the corresponding admissible trajectory of the system \eqref{control} at time $T$. Then $y(t)$ is an extremal of the Lagrangian $L$ if and only if \emph {the corresponding control $\bar u(t)=y^{(n+1)}(t)$ is a critical point of the endpoint map $\mathcal F_{q_0, T}$ for some $T>0$} (and therefore for any $T>0$ as long the corresponding trajectory is defined on $[0, T]$), where $q_0=(0, y(0),\ldots, y^n(0), z_0)$ and $z_0$ is an arbitrary constant. Take the admissible trajectory $q(t)$ of \eqref{control} corresponding to the control $\bar u(t)$ and starting at $q_0$. Then this trajectory can be lifted to the cotangent bundle $T^*(\JT{n}\times \mathbb R)$ by choosing for any $t\in [0, T]$ an appropriately normalized covectors $p(t)\in T_{q(t)}^*(\JT{n}\times \mathbb R)$ that annihilates the image of the differential
$d\mathcal F_{q_0, t}\bigl(\bar u(\cdot)\bigr)$ of the endpoint map $\mathcal F_{q_0, t}$ at $\bar u(\cdot)$.
This lifting constitutes one of the main fundamental ideas behind the Pontryagin Maximum Principle in Optimal Control (\cite{pbgmthe},\cite{agrsach}).
As a matter of fact, the curve $\bigl(p(t), q(t)\bigr)$ is an abnormal extremal of the affine control system \eqref{control} and also of the distribution associated with the Lagrangian $L$. This establish in essence the relation between extremals of the Lagrangian and the abnormal extremals of the corresponding distributions.

More precisely, the coordinates $q=(x,y_0,\ldots y_n, z)$ in $\JT{n}\times \mathbb R$ induce the coordinate system
\begin{equation}
\label{coordcot}
(p,q)=(\lambda,\xi_0,\ldots,\xi_n,\nu;\, x, y_0,\ldots y_n, z)
\end{equation}
in $T^*(\JT{n}\times \mathbb R)$ such that the covector $p\in T_q^*(\JT{n}\times \mathbb R)$ has the form
$p=\lambda\,dx+\sum_{i=0}^n \xi_i\, dy_i+\nu\,dz$.
Define the following families of scallar functions (Hamiltonians) $H_u$ on $T^*(\JT{n}\times \mathbb R)$:
\begin{equation}
\label{Hamu}
H_u(p,q)=\lambda+\sum_{i=0}^{n-1}\xi_i y_{i+1}+\xi_n u+\nu f(x,y,y',\dots,y^{(n)})
\end{equation}
According to the weak form of the Pontryagin Maximum Principle (where the maximality condition is replaced by the stationarity condition) on has the following
\begin{prop}
\label{abnprop}
A function $y(t)$ is an extremal of the Lagrangian $L$ if and only if for the admissible trajectory $q(t)$ of \eqref{control} corresponding to the control $\bar u(t)=y^{(n+1)}(t)$ and starting at
the point $q_0=(0, y(0),\ldots, y^{(n)}(0), z_0)$, where $z_0$ is an arbitrary constant, there exists a curve of nonzero covectors
$p(t)\in  T^*_{q(t)}(\JT{n}\times \mathbb R)$ such that
\begin{eqnarray}
&\frac{\partial}{\partial u}H_u(p(t),q(t))|_{u=\bar u(t)}=0 \,\,\text{a.e.}\,\, \Leftrightarrow \,\,\xi_n(t)\equiv 0 & \text{(the stationarity condition)} \label{max}\\
&H_{\bar u(t)}(p(t),q(t))\equiv 0 & \text{(the transversality condition)}\label{trans}\\
&\dot p(t)=-\frac{\partial}{\partial q}H_{\bar u(t)}(p(t),q(t)) &\text{(the adjoint equation)}\label{adj}
\end{eqnarray}
\end{prop}
Note that another part of the Hamiltonian system
\begin{equation}
\label{traj}
\dot q(t)=\frac{\partial}{\partial p}H_{\bar u(t)}(p(t),q(t))
\end{equation}
 is exactly the system \eqref{control} with $u(t)=\bar u(t)$ i.e. it holds automatically. So equations \eqref{max}-\eqref{traj} can be considered as the \emph{Hamiltonian form of the Euler Lagrange equation.}. The curve $\bigl (p(t),q(t))\subset
 T^*\bigl(\JT{n}\times \mathbb R\bigr)$ satisfying  Proposition~\ref{abnprop} is called an \emph{abnormal extremal of affine control system \eqref{control}}. The term ``abnormal'' comes again from the Pontryagin Maximum Principle applied to a functional defined on the set of admissible trajectories of system \eqref{control}: abnormal extremals are exactly the Pontryagin extremals of this problem with
vanishing Lagrange multiplier near the functional (\cite{pbgmthe,agrsach}). Roughly speaking, the extremals of our original
variational problem given by the Lagrangian $L$ become abnormal extremals of the system \eqref{control}, because we include the Lagrangian $L$ into this system so that it appears as a part of the constraints.

Let us analyze the equations \eqref{adj} in coordinates \eqref{coordcot}. First of all, since the Hamiltonians \eqref{Hamu} do not depend on $z$ we have $\dot\nu=0$ i.e. $\nu$ is constant along an abnormal extremal. If $\nu\equiv 0$, then from other equations of \eqref{adj} and equation \eqref{trans} it follows that $p(t)\equiv 0$ but $p(t)$ can not vanish by Proposition~\ref{abnprop}. So the case $\nu\equiv 0$ is impossible.
Now assume that $\nu\neq 0$. From the homogeneity of the equations \eqref{adj} with respect to $p$ it follows that it is enough to consider the case when $\nu\equiv -1$.
Then, combining the stationarity condition \eqref{max} with the equation from \eqref{adj} regarding $\dot \xi_n$ we will get that
\begin{equation}
\label{Legn}
\xi_{n-1}=f_{y_n}
\end{equation}
Writing equations for others $\xi_j$ from \eqref{adj} we get
\begin{equation}
\label{xidot}
\begin{split}
\dot \xi_0 &= f_{y_0};\\
\dot \xi_j &= f_{y_j}-\xi_{j-1},\quad j=1,\dots,n-1;\\
\end{split}
\end{equation}
Combining \eqref{Legn}  and the equation in \eqref{xidot} corresponding $j=n-1$ we get $\xi_{n-2} = f_{n-1} - \frac{d}{dx}(f_n)$. Then using the second line of \eqref{xidot} by induction with respect to $j$ in the decreasing order, we get
\begin{equation}
\label{Legj}
\xi_{j-1}= \sum_{k=j}^n (-1)^{k-j}\frac{d^{k-j}}{dx^{k-j}}(f_{y_k}),\quad 1\leq j\leq n-1.
\end{equation}
Finally substituting \eqref{Legj} with $j=1$ into the first line of \eqref{xidot} we get the Euler-Lagrange equation \eqref{EL}, as expected.

If $X$ is a vector field without stationary points or a line disribution, denote by $\Fol(X)$ the one dimensional foliation of integral curves of $X$.
Consider a codimension $4$ submanifold $\mathcal H$ of $T^*\bigl (\JT{n}\times \mathbb R\bigr)$ given
by equations \eqref{max}, \eqref{trans}, \eqref{Legn}, and $\nu=-1$. It is foliated by abnormal extremals of system \eqref{control}
with $\nu=-1$.  Besides by constructions the group of translations along $z$-axis,  $z\mapsto z+c$, preserves
this foliation. Therefore this foliation induces the one-dimensional foliation on the quotient manifold
$\mathcal H/\Fol\left(\frac{\partial}{\partial z}\right)$ by the foliation $\Fol\left(\frac{\partial}{\partial z}\right)$ or , equivalently, on the manifold of the orbits of the group
of these translations. The tuple $(x,y_0,\ldots,y_n, \xi_0,\ldots, \xi_{n-2})$ constitute a coordinate system on
the manifold $\mathcal H/\Fol\left(\frac{\partial}{\partial z}\right)$. On the other hand, the Euler--Lagrange equation \eqref{EL} defines a codimension one submanifold $\E(L)$ of $\JT{2n}$ foliated by a one-
dimensional foliation of prolongations of its solutions to $\JT{2n}$, and the tuple $(x,y_0,\ldots,y_{2n-1})$
constitute a coordinate system on  $\E(L)$. This foliation is called the \emph{foliation of solutions} of the Euler-Lagrange equation.

By above, the map $\mathfrak L\colon \E(L)\mapsto \mathcal H/\Fol\left(\frac{\partial}{\partial z}\right)$, defined by
\begin{equation}
\label{Legcoord}
(x,y_0,\ldots,y_{2n-1})\mapsto (x,y_0,\ldots,y_n, \xi_0,\ldots, \xi_{n-2}),
\end{equation}
 with $\xi_j$ satisfying \eqref{Legj}, sends the one-dimensional foliation on $\mathcal {EL}$ to the one-dimensional foliation on $\mathcal H/\Fol\left(\frac{\partial}{\partial z}\right)$. In other words, this map transforms the extremals of our variational problem obtained in the Lagrangian form to the extremal obtained in the Hamiltonian form. Therefore we call it the \emph {(generalized) Legendre transform}. Note that the Legendre transform depend on the choice of coordinates on $\JT{0}=\R^2$, which induces the coordinates on $\JT{2n}$). Once we use the Legendre transform in the sequel it will mean that such choice is already done.

\subsection{Abnormal extremals of rank 2 distributions}
\label{abnsec}
Now we are going to describe abnormal extremals for a distribution $D$ on a manifold $M$. We shall use more geometric language.
Let $\pi: T^*M\mapsto M$ be the canonical projection.
For any
$\lambda\in T^*M$, $\lambda=(p,q)$, $q\in M$, $p\in
T_q^*M$, let $\mathfrak{s}(\lambda)(\cdot)=p(\pi_*\cdot)$
be the tautological Liouville $1$-form and $\sigma=d\mathfrak {s}$
be the standard symplectic structure on $T^*M$.
Denote
by $(D^l)^{\perp}\subset T^*M$ the annihilator of the $l$th power
$D^l$, namely
\begin{equation}
\label{annihil} (D^l)^{\perp}= \{(q,p)\in T^*M:\,\, p\cdot
v=0\,\,\forall v\in D^l(q)\}.
\end{equation}
Finally let $\mathcal S_0$ be the zero section of $T^*M$.  With this
notation the Pontryagin Maximum Principle in the coordinate-free form
(\cite{agrsach}) implies immediately the following description of abnormal extremals of the distribution $D$:

\begin{defin} An absolutely continuous curve $\gamma\subset T^*M$ is an
abnormal extremal of a distribution $D$ if the following two conditions holds:
\begin{itemize}
\item [1.]$\gamma\subset D^\perp\backslash {\mathcal S_0}$,
\item [2.]$\dot\gamma(t)$ belongs to ${\rm
Ker}\bigl(\sigma\bigl|_{D^\perp}\bigr.\bigr)$, i.e., to the kernel
of the restriction of the canonical symplectic form $\sigma$ to the
annihilator $D^\perp$ of $D$.
\end{itemize}
\end{defin}

From now on we will consider only rank $2$-distributions. From direct
computations \cite[Proposition 2.2]{zel99} it follows that ${\rm
Ker}\bigl(\sigma(\lambda)\bigl|_{D^\perp}\bigr.\bigr)\neq 0$ if and only if
$\lambda\in (D^2)^\perp$. This implies the following
characterization of abnormal extremals of rank 2 distribution.

\begin{prop}
An absolutely continuous curve  $\gamma\subset T^*M$ is abnormal
extremal of a rank $2$ distribution $D$ with $\dim D^2=3$ if and only if the
following two conditions holds
\begin{itemize}
\item [1.]$\Gamma\subset (D^2)^\perp\backslash {\mathcal S_0}$,
\item [2.]$\dot\Gamma(t)$ belongs to $\Ker\bigl(\sigma\bigl|_{(D^2)^\perp}\bigr.\bigr)$, i.e., to the
kernel of the restriction of the canonical symplectic form $\sigma$
to the annihilator $(D^2)^\perp$ of $D^2$.
\end{itemize}
\end{prop}

Further, if $\lambda\in (D^2)^\perp\backslash (D^3)^\perp$ then
$\Ker \bigl(\sigma\bigl|_{(D^2)^\perp}\bigr.\bigr)$ is
one-dimensional. These kernels form a special line distribution on
$\lambda\in (D^2)^\perp\backslash (D^3)^\perp$, which will be
denoted by $\widetilde {\mathcal C}$, called the \emph{characteristic distribution}. The abnormal extremals of $D$, lying in $\lambda\in
(D^2)^\perp\backslash (D^3)^\perp$, are exactly the integral curves
of the line distribution $\mathcal C$ (in some literature these abnormal
extremals are called regular).

\begin{rem}
\label{liftD}
For any $\lambda\in (D^2)^\perp\backslash (D^3)^\perp$ let
\begin{equation}
\label{Jdef}
{\widetilde{\mathcal J}}(\lambda)=
\{v\in T_{\lambda}(D^2)^\perp:\,\pi_*\,v\in
D(\pi\bigl(\lambda)\bigr)\}.
\end{equation}
A simple count shows that $\dim {\widetilde{\mathcal J}}(\lambda)=n+2$.
Then from constructions it follows immediately that the restriction of the form $\sigma(\lambda)$ to $\mathcal J (\lambda)$ is identically equal to zero.
\end{rem}

Now consider the distribution $D$ on $M=\JT{n}\times \mathbb R$ associated with the Lagrangian  $L=f(x,y,y',\dots,y^{(n)})\,dx$ with $\frac{\partial^2 f}{\partial y_n^2}\ne0$.
We would like to rewrite  the constructions of the end of the previous subsection in more geometric form.
First of all in the considered case $D^\perp$ is a corank $2$ submanifold of $T^*M$ given by equations \eqref{max} and \eqref{trans}, $(D^2)^\perp$ is a corank $1$ submanifold of $D^\perp$ satisfying in additional equation \eqref{Legn}, and $(D^3)^\perp$ is a corank 2 subdistribution of $(D^2)^\perp$ satisfying to additional equations $\nu=0$ and $\xi_{n-2}=0$. The submanifold $\mathcal H$ introduced in the previous subsection is equal to $\{H_Z=-1\}\cap \bigl(D^2)^{\perp}
$ and the abnormal extremals of the distribution $D$ lying on $\mathcal H$ coincide (as unparametrized curves) with the abnormal extremals of system \eqref{control} having $\nu=-1$.

Further, given a vector field $X$ on $M$ denote by  $H_X: T^*M\to \mathbb R$ the corresponding quasi-impulse  $$H_X(p,q)=p\bigl(X(q)\bigr),\quad  q\in M, p\in T^*_qM,$$ and by $\vec H_X$  the corresponding Hamiltonian vector field on $T^*M$, i.e. the vector field satisfying $i_{\vec H_X}\sigma=-dH_X$.
It is clear that if $X$
is an infinitesimal symmetry of the distribution $D$, then  the flow $e^{t\vec H_X}$, generated by $\vec H_X$,  sends an abnormal extremal of $D$ to an abnormal extremal of $D$. Moreover, any abnormal extremal lies on a level set of the function $H_X$.
In particular, let as in Lemma \ref{lem1} $Z$ be the infinitesimal symmetry of $D$ lying in $D^3$.
 The distribution $\widetilde{\mathcal C}$ induces a rank 1 distribution $\bar {\mathcal C}$ on the quotient manifold
 $\mathcal H/\Fol(\vec H_Z)$, where as before $\Fol(H_Z)$ is the foliation of integral curves of the field $\vec H_Z$.
 A Legendre transform  $\mathfrak L\colon \E(L)\rightarrow \mathcal H/\Fol(\vec H_Z)$, defined in the previous subsection, sends the one-dimensional foliation of solutions of Euler-Lagrange equations to the one-dimensional  foliation of
the integral curves of the distribution $\bar {\mathcal C}$.

\begin{rem}
\label{fiberem}
The Legendre transform $\mathfrak L$  satisfies another important property. To describe it in geometric terms let $\pi_{i,j}:\JT{i}\rightarrow \JT{j}$ , where $i>j$, and $$\bar \pi: T^*(\JT{n}\times \mathbb R)\rightarrow \JT{n}\times \mathbb R $$
denote the canonical projections. The mapping $\bar \pi$ induces the mapping  $$\pi_Z\colon T^*(\JT{n}\times \mathbb R)/\Fol(\vec H_Z)\mapsto(\JT{n}\times \mathbb R)/\Fol(Z)\sim \JT{n}$$ in the obvious way. If the submanifolds $\E(L)$ and  $\mathcal H/\Fol(\vec H_Z)$ are considered as fiber bundles over  $\JT{n}$ with the projections $\pi_{2n,n}|_{_{\E(L)}}$ and
$\pi_Z|_{\mathcal H/e^{t\vec H_Z}}$, respectively, then from \eqref{Legcoord} it follows immediately that the Legendre transform $\mathfrak L$ is fiberwise mapping over the identity on the base manifold $\JT{n}$.
\end{rem}

\begin{rem}
\label{sympind}
The tautological Liouville $1$-form
$\mathfrak{s}$ and the standard symplectic structure $\sigma$ on $T^*(\JT{n}\times \mathbb R)$ induce the $1$-form $\bar{\mathfrak s}$ and the closed $2$-form $\bar\sigma=d\bar{\mathfrak s}$, respectively,  on $\mathcal H/\Fol(\vec H_Z)$ . By constructions, the rank $1$ distribution $\bar {\mathcal C}$ satisfies $\bar {\mathcal C}=\rm{Ker}\,\bar\sigma$. Besides, using condition \eqref{trans} it is easy to show that
\begin{equation}
\label{1formcoord}
\bar{\mathfrak s}=L+\sum_{i=1}^{n-2} \xi_i \theta_i+f_{y_n}\theta_{n-1}
\end{equation}
in the coordinates $(x,y_0,\ldots,y_n, \xi_0,\ldots, \xi_{n-2})$ on $\mathcal H/\Fol(\vec H_Z)$, where, as before, $\theta_i=dy_i-y_{i+1}\,dx$.
Finally, for any $\lambda\in\mathcal H/\Fol(\vec H_Z)$ we denote
\begin{equation}
\label{barJdef}
\bar {\mathcal J}(\lambda)=\{v\in T_{\lambda}\mathcal H/\Fol(\vec H_Z):\,(\pi_Z)_*\,v\in
D(\pi\bigl(\lambda)\bigr)\}.
\end{equation}
Then from the last sentence of Remark \ref{liftD} the restriction of the form $\bar\sigma(\lambda)$ to the subspace $\bar {\mathcal J}(\lambda)$ is identically equal to zero.
\end{rem}

\section{Linearization of variational ODEs and Jacobi curves of rank 2 distribution}
\label{linsect}
\setcounter{equation}{0}
\setcounter{thm}{0}
\setcounter{lem}{0}
\setcounter{prop}{0}
\setcounter{cor}{0}
\setcounter{rem}{0}

Let us outline the content of this section.
As was mentioned before both in the equivalence problem for variational ODEs and in the equivalence problem for rank 2 distributions self-dual curves in a projective space play a crucial role. They appear via the linearization along the "flow "of solutions in the first case and along the "flow" of abnormal extremals in the second case. Using the Legendre transform introduced in subsection \ref{leg} we show that the self-dual curves in a projective space obtained by the linearization
along a solution of the Euler-Lagrange equations of a Lagrangian and  by the linearization
along the corresponding abnormal extremal of the associated rank 2 distribution are actually isomorphic.
This observation leads to the description  of the fundamental system of invariants for rank 2 distributions associated with Lagrangians given in the next section.
\subsection{General linearization procedure}
\label{gen}
Let us first clarify what do we mean by the linearization procedure in a general geometric setting.
Let $\mathcal M$ be an arbitrary smooth
manifold, let $\mathcal G$ and  $\mathcal V$ be a pair of vector distributions on $\mathcal M$ of rank $l$ and $k$, respectively,
where one of them, say $ \mathcal G$, is integrable and $\mathcal V\cap \mathcal G$ is a distribution of rank $r$.

Similarly to above, let $\Fol(\mathcal G)$ be a foliation of $\mathcal M$ by maximal integral submanifolds of $\mathcal G$. Then we can define the \emph{linearization of the distribution $\mathcal V$ along the foliation $\Fol(\mathcal G)$} in the following way. Let, as above, $\Fol(\mathcal G)$ be a foliation of $\mathcal M$ by
maximal integral submanifolds of $\mathcal G$. Locally we can assume that there exists a
quotient manifold $\mathcal M/\Fol(\mathcal G)$, whose points are leaves of $\mathcal F(\mathcal G)$. Let $\Gamma$ by any such leaf. Then we define the map $\phi$ of $\Gamma$ into the Grassmannian
$\mathrm{Gr}_{k-r}(T_\Gamma\bigl(\mathcal M/\Fol(\mathcal G)\bigr)$ of $(k-r)$-dimensional subspaces of  $\bigl(\mathcal M/\Fol(\mathcal G)\bigr)$ or, under additional regularity assumptions, an $l$-dimensional submanifold of  ${\rm Gr}_{k-r}(T_\Gamma\bigl(\mathcal M/\Fol(\mathcal G)\bigr))$ as follows: $\phi(x)=\mathrm {pr}_*(\mathcal V_x)$, $x\in \Gamma$, where $\mathrm {pr}\colon
\mathcal M\to \mathcal M/\Fol(\mathcal G)$ is a natural projection.
The map $\phi$ or its image in  $\rm {Gr}_{k-r}(T_\Gamma\bigl(\mathcal M/\Fol(\mathcal G)\bigr))$ is called the \emph{linearization of the distribution $\mathcal V$ along the foliation $\Fol (\mathcal G)$ at the leaf $\Gamma$} or
the \emph {linearization of the distribution $\mathcal V$ along the leaf $\Gamma$ (of $\Fol (\mathcal G)$}). In the cases under consideration $l=1$ so that the linearizations are  curves in projective spaces.

The main idea of using the linearization procedure in the equivalence problem for the structures given by the pair of distribution $(\mathcal V,\mathcal G)$ on $\mathcal M$ with respect to the action of group of diffeomorphisms of $M$ is that it allows to construct the invariants of such structures from invariants of submanifold in an appropriate Grassmannian with respect to the natural action of the General Linear Group on this Grassmannian.

\subsection{Linearization procedure for ODEs}
\label{ODElin}
As in subsection \ref{ODEsubsec} an ODE of order $N+1$, resolved with respect
to the highest derivative, is given by a hypersurface $\E$ in the jet space
$\JT{N+1}$ so that the restriction of the natural projection
$\pi_{N+1,N}\colon \JT{N+1}\to\JT{N}$ to the hypersurface $\E$ is a diffeomorphism.
Further, the Cartan distribution of $\JT{N+1}$ (i.e. the rank 2 distribution defined by contact forms $\theta_i=dy_i -y_{i+1}dx,\quad i=0,\dots,N$, in the standard coordinates $(x,y_0,\ldots, y_{N+1})$ in $\JT{N+1}$) defines the line distribution $\mathcal S$ on $\E$. This distribution is obtained by the intersection of the Cartan distribution with the tangent space to $\E$ at every point of $\E$.
Note that the corresponding foliation $\Fol(S)$ is the foliation of solutions of our ODE (more precisely, the foliation of the prolongations of the solutions to $\JT{N}$).
If the hypersurface $\E$ has the form \eqref{ODE} in coordinates $(x,y_0,\ldots, y_{N+1})$ on $\JT{N+1}$, then in the coordinates $(x,y_0,\ldots, y_{N})$ on~$\E$:
\begin{equation}
\label{Scoord}
S=\left\langle \frac{\partial}{\partial x} + \sum_{i=1}^{N-1}y_{i+1} \frac{\partial}{\partial y_i}+F(x,y_0,y_1,\dots,y_{N})\frac{\partial\hspace{.15in}}{\partial y_{N}}\right\rangle.
\end{equation}
The distribution $\mathcal S$ will play the role of the distribution $\mathcal G$ from the previous subsection.

Further let, as before,  $\pi_{N+1, N-i}\colon\JT{N+1}\mapsto\JT{N-i}$ be the canonical projection. For any $\varepsilon\in\E$ we can define the filtration $\{V^i_\varepsilon\}_{i=0}^{N}$ of $T_\varepsilon \E$ as follows:
\begin{equation}
\label{Vicoord}
V^i_\varepsilon=\ker d_\varepsilon\pi_{N+1, N-i} \cap T_\varepsilon \E.
\end{equation}
Then $V^i$ is a rank $i$ distribution on $\E$. In the coordinates $(x,y_0,\ldots, y_N)$ on $\E$ we have
\begin{equation}
\label{Vicoord1}
V^i=\left\langle \frac{\partial}{\partial y_{N-i+1}},\dots, \frac{\partial}{\partial y_{N}} \right\rangle.
\end{equation}

Let $\mathrm{Sol}$ denote the quotient manifold $\mathcal \E/\Fol(\mathcal S)$, i.e. the manifold of solutions of the equation $\E$. Fix a point $\Gamma \in \rm {Sol}$. In other words, $\Gamma$ is a leaf of $\Fol(\mathcal S)$ or a solution of the equation $\E$. Consider the linearization $\mathrm {Lin}^i_
\Gamma$ of the distribution $V^i$ along $\Gamma$. It is a curve in $\mathrm {Gr}_i\bigl(T_\Gamma \mathrm {Sol}\bigr)$. In particular, $\mathrm {Lin}^1_\Gamma$
is a curve in the projective space $\mathbb P(T_\Gamma\rm{Sol})$. Moreover, if $\Gamma$ is considered as the leaf of $\Fol(\mathcal S)$, then
from \eqref{Scoord} and \eqref{Vicoord} one gets immediately the following
\begin{lem}
\label{osclem}
For any $\varepsilon\in \Gamma$   the $i$-dimensional subspace $\rm {Lin}^i_\Gamma(\varepsilon)$ of $T_\Gamma\rm{Sol}$ is exactly the $i$-th osculating space of the curve $\rm {Lin}^1_\Gamma$ at
$\varepsilon$ (as defined in \eqref{osc}).
\end{lem}

The Wilczynski invariants of the linearizations  $\rm {Lin}^1_\Gamma$ taken for every solution $\Gamma$ define the invariants of the original ODE $\E$ under the group of contact transformations (see~\cite{dou}). We call these invariants the \emph{generalized Wilczinski invariants of $\E$} and denote them also by $W_i$, $i=3,\dots,N+1$.

\subsection{The case of variational ODEs}
Now assume that $\E$ is a variational ODE,  $\E=\E(L)$ for some Lagrangian $L$. Let $\bar\sigma$ be the closed $2$-form on $H/\Fol(\vec H_Z)$ introduced  in Remark \ref{sympind} and let $\mathfrak L\colon\E(L)\rightarrow \mathcal H/\Fol(\vec H_Z)$ be a Legendre transform. Then we can define the closed $2$-form $\omega$ on $\E(L)$ as follows:
\begin{equation}
\label{omega1}
\omega=\mathfrak L^*\bar\sigma
\end{equation}

Note that from relation \eqref{Vicoord1} and Remark \ref{fiberem} it follows that the distribution $\bar {\mathcal J}$ defined by \eqref{barJdef} satisfies
\begin{equation}
\label{LegJ}
\bar {\mathcal J}=\mathfrak L_* (V_n)\oplus \bar {\mathcal C}.
\end{equation}
From this and Remark \ref{sympind}  it follows that the form $\omega$ satisfies the following two properties:
\begin{enumerate}
\item $\mathrm{Ker}\,\omega=\mathcal S$, where, as before, $\mathcal S$ is the rank 1 distribution generating the foliation of solutions of $\E(L)$;

\item The restriction of $\omega$ to the distribution $V^n$ vanishes.
\end{enumerate}

From property (1) it follows that $\omega$ induces the symplectic form $\bar\omega$ on the manifolds ${\rm Sol}$ of solutions of $\E(L)$. Moreover from the property (2) it follows that

\begin{enumerate}
\item[(2')]
For any $\Gamma\in \E(L)$ the linearization ${\rm Lin}_\Gamma^n$ of the distribution $V^n$ along $G$ is the curve of Lagrangian subspaces with respect to the symplectic form $\omega(\Gamma)$.
\end{enumerate}
From item (2) of Proposition \ref{selfdual} we get immediately  the following
\begin{cor}
\label{seldual2}
For a variational ODE
the linearizations  ${\rm Lin}_\Gamma^1$ along any solution $\Gamma$ is a self-dual curve in the corresponding projective space.
\end{cor}

Further, as an immediate consequence of item (1) of Proposition  \ref{selfdual} and the fact that $\omega$ is closed, we get
\begin{prop}\label{sympexist}
For a variational ODE $\E(L)$ there exists a unique, up to a constant nonzero factor, closed $2$-form $\omega$ on $\E(L)$ satisfying conditions \text{(1)} and \text{(2)} above  or, equivalently, a unique, up to a constant nonzero factor, symplectic structure $\bar\omega$ on the manifold of solutions ${\rm Sol}$, satisfying condition \text{(2')} above .
\end{prop}

\begin{rem}
It is easy to see that this symplectic strucutre $\bar\omega$ from Proposition~\ref{sympexist} is exactly the
2-form $\omega$ given by ~\eqref{omega_and} and prescribed by Theorem~2.6 of~\cite{and}. Indeed,
the condition $d\omega = 0$ implies $d_H\omega = 0$. In particular, this means that $\omega$ projects to the
solution space of the equation $\E(L)$. Proposition~\label{sympexits} gives an alternative construction for this
2-form based on the generalized Legendge transform.
\end{rem}

Now we are ready to prove the following Theorem (which is also proved in \cite[Theorem 5.1]{fels} in the case $n=2$ and in \cite[Corollary 2.6]{juras} for $n\ge3$):

\begin{thm}\label{cor1}
Lagrangians
are equivalent if and only if their Euler--Lagrange equations are contact equivalent.
\end{thm}
\begin{proof}
If Lagrangians  $L_1$ and $L_2$ are equivalent then directly from \eqref{Lequiv} it follows that their Euler-Larange equations are contact equivalent.

In the other direction let $\psi$ be the mapping establishing the equivalence of two Euler--Lagrange
equations $\E_1$ and $\E_2$. Let $\omega_1$ and $\omega_2$ be the closed $2$-forms from Proposition~\ref{sympexist}, corresponding to
$\E_1$ and $\E_2$, respectively. Then by this Proposition there exists a constant $\alpha\neq 0$ such that
\begin{equation}
\label{omega12}
\psi^*\omega_2=\alpha\omega_1.
\end{equation}

Now assume that $\bar{\mathfrak s_1}$ and $\bar{\mathfrak s_2}$ are $1$-forms from Remark \ref{sympind}
corresponding to Lagrangians $L_1$ and $L_2$, respectively. Assume that $\mathfrak L_i$ are the corresponding Legendre transforms and let
$\rho_i=(\mathfrak L_i)^*\mathfrak s_i$, $i=1,2$. Then by construction $\omega_i=d\rho_i$. Hence from \eqref{omega12} there exists a function
$\mu$ on $\E_1$ such that
\begin{equation*}
\psi^*\rho_2=\alpha\rho_1 +d\mu.
\end{equation*}
Taking into account the coordinate expressions for the forms $\bar{\mathfrak s}_i$ given by \eqref{1formcoord}, we get immediately that the last relation is equivalent to \eqref{Lequiv}, i.e. the Lagrangians $L_1$ and $L_2$ are equivalent.
\end{proof}

\subsection{Linearization procedure for rank 2 distributions: Jacobi curves.}
\label{lindist}
The presentation of this subsection is rather closed to our previous works \cite{dz1,dz2} but it is considered here in relation with the linearization procedure of ODE's from subsection \eqref{ODElin}.
Let $D$ be an arbitrary rank 2 vector distribution on an $(n+3)$-dimensional
manifold $M$. We shall assume that $D$ is completely non-holonomic and that
$\dim D^3 = 5$. Here it is more convenient to work with the projectivization of  $\mathbb PT^*M$ rather than with $T^*M$.
Here $\mathbb PT^*M$ is the fiber bundle over $M$ with the fibers that are the projectivizations of the fibers of $T^*M$.
The canonical projection $\Pi\colon T^*M\rightarrow \mathbb PT^*M$ sends the characteristic distribution
$\widetilde{\mathcal C}$ of $(D^2)^\perp\backslash (D^3)^\perp$ to the line distribution $\mathcal C$ on
$\mathbb P(D^2)^\perp\backslash \mathbb P(D^3)^\perp$, which will be also called the \emph{characteristic distribution}
of the latter manifold. The manifold $\mathbb P(D^2)^\perp\backslash \mathbb P(D^3)^\perp$ and the distribution $\mathcal C$ play the role of $\mathcal M$ and $\mathcal G$, respectively, from the general linearization procedure of subsection \ref{gen}.

Further note that the corank 1 distribution on $T^*M\backslash S_0$ annihilating the tautological Liouville form $\mathfrak s$ on $T^*M$ induces a contact distribution on $ \mathbb P T^*M$, which in turns induces the even-contact (quasi-contact) distribution $\Delta$
on $\mathbb P(D^2)^\perp\backslash \mathbb P(D^3)^\perp$. The characteristic line distribution $\mathcal C$ is exactly the Cauchy characteristic distribution of $\Delta$, i.e. it is the maximal subdistribution of $\Delta$ such that
\begin{equation}
\label{Cauchy}
[C,\Delta]\subset \Delta.
\end{equation}
Now let $\widetilde{\mathcal J}$ be as in \eqref{Jdef}. Let
${\mathcal J}(\lambda)=\Pi_*\widetilde{\mathcal J}$
and define a sequence of subspaces ${\mathcal
J}^{(i)}(\lambda)$, $\lambda\in \mathbb P(D^2)^\perp\backslash
\mathbb P(D^3)^\perp$, by the following recursive formulas:
$${\mathcal J}^{(i)}(\lambda)={\mathcal J}^{(i-1)}(\lambda)+[\mathcal C, \mathcal J^{i-1}](\lambda).$$
By \cite[Proposition 3.1]{zel06},
we have
${\rm dim}\,{\mathcal J}^{(1)}(\lambda)-{\rm dim}\,
{\mathcal J}(\lambda) =1$,
which implies easily that
$${\rm dim}\, {\mathcal J}^{(i)}(\lambda)- {\rm dim}\,
{\mathcal J}^{(i-1)}(\lambda)\leq 1, \quad \forall i\in {\mathbb
N}.$$
Besides, from \eqref{Cauchy} it follows that ${\mathcal
J}^{(i)}\subset \Delta$ for all
natural $i$. Simple counting of dimensions implies that $rank \Delta=2n+1$ so that  ${\rm dim}\, {\mathcal
J}^{(i)}(\lambda)\leq 2n+1$.

Note that for any $\lambda\in \mathbb P(D^2)^\perp\backslash \mathbb P(D^3)^\perp$  the subspace $\Delta(\lambda)$ is equipped canonically, up to a constant nonzero factor, with a skew-symmetric form with the kernel equal to $\mathcal C(\lambda)$: for this take the restriction to $\Delta(\lambda)$ of the differential of any $1$-form annihilating $\Delta$. Given a subspace $W$ of $\Delta(\lambda)$ denote by $W^\angle$ the skew-symmetric complement of $W$ with respect to this form. Note that by Remark \ref{liftD} $(J^{(0)})^\angle=J^{(0)}$
 Then set
\begin{equation}
\label{contr} {\mathcal J}^{(-i)}(\lambda)= \bigl({\mathcal J}^{(-i)}(\lambda)\bigr)^\angle.
\end{equation}
The sequence of subspaces $\{\mathcal J^{(i)}(\lambda)\}_{i\in \mathbb Z}$ defines the filtration of $\Delta(\lambda)$.

Further,
define the following two integer-valued functions:
\begin{equation*}
\nu(\lambda)=\min\{i\in {\mathbb N}: {\mathcal J}^{(i+1)}(\lambda)={\mathcal J}^{(i)}(\lambda)\},
\end{equation*}
\begin{equation*}
m(q)=\max\{\nu(\lambda):\lambda\in
(D^2)^\perp(q)\backslash(D^3)^\perp(q)\}, \quad q\in M.
\end{equation*}
The number $m(q)$ is called \emph{the class of distribution
$D$ at the point $q$}. By above, $1\leq m(q)\leq n$. It
is easy to show that \emph{germs of $(2,n+3)$-distributions of
the maximal class $n$
 are generic} (
see~\cite[Proposition~3.4]{zel06}).

From now on we assume that $D$ is a $(2,n+3)$-distribution of
maximal constant
class $m=n$.
Let 
${\mathcal R}=\{\lambda\in \mathbb P(D^2)^\perp \backslash
\mathbb P (D^3)^\perp \nu(\lambda)=n\}$. Then on $\lambda\in\mathcal R$ the subspaces $\mathcal J^{(i)}$ form a distribution  of rank $(i+n+1)$
for all integer $i$ between $-n$ and $n$ (in particular $\mathcal J^{(n)}= \Delta$). If we denote by $\rm {Abn}$ the quotient of manifold $\mathbb P(D^2)^\perp \backslash
\mathbb P (D^3)^\perp$ by $\Fol(\mathcal C)$ then the distribution $\Delta$ induces the distribution $\bar \Delta$ on $\rm {Abn}$ equipped with the canonical, up to a constant nonzero factor, symplectic form. Given any segment $\Upsilon$ of abnormal extremal ( a leaf of $\Fol(\mathcal C)$) consider the linearization  $J_\Upsilon^{(i)}$ of the distribution
$\mathcal J^{(i)}$ along $\Upsilon$.
It is a curve in $\rm {Gr}_{n+i}\bigl(\bar\Delta(\Upsilon)\bigr)$. In particular, $J^{(1-n)}_\Upsilon$
is a curve in the projective space $\mathbb P(\bar\Delta(\Gamma))$.
The curve $\rm J^{(1-n)}_\Gamma$ is called the \emph{Jacobi curve} along the abnormal extremal $\Gamma$. By Remark
\ref{liftD} it is the curve of Lagrangian subspaces of $\bar \Delta$.
Further, it is not hard to see (\cite{zel06}) that for any $\lambda\in \Gamma$   the $(n+i)$-dimensional subspace
$J^{(i)}_\Upsilon(\lambda)$ of $T_\Upsilon\rm{Abn}$ is exactly the $i+n-1$-st osculating space of the curve $J^{(1-n)}_\Upsilon$ at
$\lambda$ (as defined in \eqref{osc}). So, by item (2) of Proposition \ref{selfdual}, the curve $J^{(1-n)}_\Upsilon$ is self-dual.
\begin{rem}
\label{recov}
It can be shown \cite{zel05} that $J^{(1-n)}_\Upsilon$ is the only curve in $\mathbb P(\bar\Delta(\Upsilon))$ such that the Jacobi curve at $\lambda\in\Upsilon$ is its $(n-1)$-st osculating space at $\lambda$. $\Box$
\end{rem}

The Wilczynski invariants of the linearizations  $J^{(1-n)}_\Upsilon$ taken for every abnormal $\Upsilon$ define the invariants of the distribution $D$. The latter invariants are called the \emph{generalized Wilczinski invariants of $\E$}. The Jacobi curve
along the abnormal extremal is called \emph{flat} if the corresponding curve  $J^{(1-n)}_\Upsilon$ is a rational normal curve in
$\mathbb P(\bar\Delta(\Upsilon))$ or, equvalently, all Wilczinsky invariants of $J^{(1-n)}_\Upsilon$ vanishes identically.
For the maximally symmetric $(2,n+3)$-distribution of maximal class which is locally equivalent to the distribution associated with the Lagrangian $(y^{(n)})^2\,dx$  (Theorem~\ref{thmdz})  all Jacobi curves are flat or, equivalently, all generalized Wilczynski invariants vanish. The general question is
\medskip

{\bf Question} \emph {Is it true that if all generalized Wilczynski invariants of $(2, n+3)$ distribution of
maximal class vanish or , equivalently, all its Jacobi curves are flat, then the distribution is locally equivalent to
the distribution associated with the Lagrangian $(y^{(n)})^2\,dx$?}
\medskip

Since, as was shown in \cite{zel06c}, in the case $n=2$ the generalized Wilczynski invariant coincides with Cartan's covariant binary
biquadratic form  introduced  in \cite{cartan10}, then the fact that this form is the fundamental invariant of a $(2,5)$-distribution
(proved in \cite{cartan10} as well) gives the positive answer to our question in this case.
We show in the next section that for $n\geq 3$ the answer to this question is positive if we restrict ourselves
to distributions associated with Lagrangians (see Theorem \ref{Wilthm} below). Note that this class of distributions is rather restrictive and the general question for $n\geq 3$ remains open.

Assume that the distribution $D$ is associated with some Lagrangian $L$ with $\frac{\partial^2 f}{\partial y_n^2}\ne0$. Then in the notations of section the manifold  $\mathcal H$ can be identified with an open subset of $\mathcal R$. Let $\phi_z\colon
\mathcal H_Z\rightarrow \mathcal H_Z/\Fol(\vec H_z)$ be the canonical projection.
Then, as was already mentioned in subsection \ref{abnsec} the characteristic distribution $\widetilde {\mathcal C}$ is reduced to the  line distribution $\bar C=(\phi_Z)_*C$ on
$\mathcal H/\Fol(\vec H_z)$. Moreover, the filtration $\{\mathcal J^{(i)}\}_{i\in \mathbb Z}$ on $\mathcal H$ induces the filtration  $\{\bar{\mathcal J}^{(i)}\}_{i\in \mathbb Z}$ on $\mathcal H/\Fol(\vec H_z)$. Note that in this notation the distribution $\bar {\mathcal J}^{(0)}$ coincides with $\bar {\mathcal J}$ defined by \eqref{barJdef}. If $\Upsilon_1$ and
$\Upsilon_2$ are two abnormal extremals on $\mathcal H$ such that $\phi_Z(\Upsilon_1)=\phi_Z(\Upsilon_2)$.
then the curves  $J^{(1-n)}_{\Upsilon_1}$ and $J^{(1-n)}_{\Upsilon_2}$ coincide up to a projective transformation. As a matter of fact they coincide, up to a projective transformation, with the linearization of the distribution $\bar{\mathcal J}^{(1-n)}$ along $\bar\Upsilon=(\phi_Z)_*\Upsilon_1$ on $\mathcal H/\Fol(\vec H_z)$.
From this, Lemma \ref{osclem}, Remark \ref{recov},  and \eqref{LegJ} we have the following

\begin{prop}
\label{equivcurve}
Given a solution $\Gamma$ of the Euler-Lagrange equation $\E(L)$ and the abnormal extremal $\Upsilon$ such that $\phi_Z(\Upsilon)=\mathfrak L(\Gamma)$, where $\mathfrak L$ is a Legendre transform, the curves $\rm {Lin}^1_\Gamma$ and $J^{(1-n)}_\Upsilon$ coincide up to a projective transformation.
\end{prop}




\section{Fundamental invariants for our equivalence problems}
\setcounter{equation}{0}
\setcounter{thm}{0}
\setcounter{lem}{0}
\setcounter{prop}{0}
\setcounter{cor}{0}
\setcounter{rem}{0}
Now we are ready to prove the following
\begin{thm}
\label{Wilthm}
For any $n\geq 3$ a $(2, n+3)$-distribution $D$ associated with a Lagrangian $f(x,y,y',\dots,y^{(n)})\,dx$ with $\frac{\partial^2 f}{\partial y_n^2}\ne0$ (or the Lagrangian $L$ itself) is locally equivalent
to the distribution associated with the Lagrangian $(y^{(n)})^2\,dx$
if and only if one of the following two equivalent
conditions is satisfied:
\begin{enumerate}
\item all Jacobi curves of the distribution $D$ are
flat;
\item generalized Wilczynski invariant of $D$
vanish identically.
\end{enumerate}
\end{thm}
The proof of the theorem immediately follows from Theorem~\ref{cor1} and the following
\begin{thm}
\label{basiclem}
Let $y^{(2n)}=F(x,y,y',\dots,y^{(2n-1)})$ be the Euler-Lagrange equation of the
Lagrangian $L=f(x,y,y',\dots,y^{(n)}) dx$. This equation is contact
equivalent to the trivial equation $y^{(2n)}=0$ if and only if all its generalized
Wilczynski invariants $W_4, W_6,\dots, W_{2n}$ vanish identically.
\end{thm}
\begin{proof}
In the sequel we will denote by $F_i$ the partial derivative $F_{y_i}$. The higher order derivatives of $F$ will be denoted in a similar way.
According to~\cite{dou}, any ordinary differential equation of order $2n$,
$n\ge3$, is trivializable if and only if all its generalized Wilczynski
invariants vanish identically and, in addition, the following conditions hold:
\begin{align}
\label{eqn3}
& F_{55}=F_{45} = 0& \qquad&\text{for }n=3;\\
\label{eqn4}
& F_{2n-1,2n-1}=F_{2n-1,2n-2}=F_{2n-2,2n-2}=0&\qquad&\text{for
}n\ge4.
\end{align}
Since our equation is variational, its generalized Wilczynski invariants of odd degree vanish automatically.
The generalized Wilczynski invariants of even degree vanish by assumption of the lemma. So, we need to prove that the above conditions also hold.

Since the equation $y^{(2n)}=F(x,y,y',\dots,y^{(2n-1)})$
is variational, then according to~\cite{and} $F$ is a polynomial in $y^{(n+1)}$, \dots, $y^{(2n-1)}$ of weighted degree $\le n$, where these derivatives have weights $1$, \dots, $n-1$ respectively.
In particular, we see that the polynomials $\big(y^{(2n-1)}\big)^2$,
$y^{(2n-1)}y^{(2n-2)}$, $\big(y^{(2n-2)}\big)^2$  have weighted degree $2n-2$,
$2n-3$ and $2n-4$ respectively. Assume that $n\ge5$. Then $2n-4>n$, and these
terms can not appear in $F$. Thus, the condition~\eqref{eqn4} holds
automatically for $n\ge5$. So, it remains to consider only the cases $n=3,4$.

Let $n=3$. Then the term $y^{(5)}$ has weighted degree $2$, while the function
$f$ is of weighted degree $\le3$. So, we see that $F_{55}=0$. Let us prove that
the condition $W_4=0$ implies also that $F_{45}=0$. The direct computation
shows that $I=F_{45}=-3f_{333}/f_{33}$. Let us denote $W_4$ simply by $W$.
From~\cite{dou} we have:
\begin{multline*}
W =
-\frac{5}{36}F_{5xxx}+\frac{2}{21}F_4F_{55}-\frac{5}{12}F_{3x}+\frac{1}{3}F_{4xx}\\
+\frac{5}{18}F_5F_{5xx}+\frac{5}{36}F_3F_5-\frac{5}{21}F_5^2F_{5x}-\frac{37}{126}F_4F_{5x}\\
+\frac{5}{252}F_5^4+\frac{37}{630}F_4^2+\frac{25}{84}F_{5x}^2+\frac{5}{18}F_2-\frac{5}{18}F_5F_{4x},
\end{multline*}
where $F_i$ denotes the partial derivative by $y^{(i)}$ and $F_x$ denotes the
total derivative. Then the direct computation shows that:
\begin{align*}
W_{55} &= \frac{1}{35}\frac{57 f_{333}^2-35f_{33}f_{3333}}{f_{33}^2};\\
W_{355} &= -\frac{1}{35}\frac{35f_{33}f_{33333} - 149
f_{33}f_{333}f_{3333}+144f_{333}^2}{f_{33}^3};
\\
W_{445} &= -\frac{2}{35}\frac{35f_{33}f_{33333} - 162
f_{33}f_{333}f_{3333}+135f_{333}^2}{f_{33}^3},
\end{align*}
and we have the following syzygy:
\[
210W_{355}-105W_{445}+26IW_{55}-\frac{4}{105}I^3 = 0.
\]
In particular, if $W=0$ we get $I^3=0$ and hence $I=0$. In particular, the
function $f$ is actually quadratic in the highest derivative. This proves the
case $n=3$.

For $n=4$, we see that $F_{77}=F_{76}=0$ due to the weighted degree argument.
So, it remains to prove that $I=F_{66}=-6f_{444}/f_{44}$ vanishes if all
generalized Wilczynski invariants vanish identically. Again, denote by $W=W_4$
the first non-trivial Wilczynski invariant for Euler-Lagrange equation. Then we
have:
\begin{multline*}
W =
\frac{35}{528}F_5F_7+\frac{49}{176}F_7F_{7xx}+\frac{7}{22}F_{6xx}-\frac{35}{176}F_7F_{6x}\\
-\frac{1127}{6336}F_7^2F_{7x}+\frac{161}{3168}F_6F_7^2+\frac{931}{3168}F_{7x}^2+\frac{7}{66}F_4\\
+\frac{47}{1584}F_6^2+\frac{1127}{101376}F_7^4-\frac{329}{1584}F_6F_{7x}-\frac{49}{264}F_{7xxx}.
\end{multline*}
Again, direct computation proves shows that:
\begin{align*}
W_{75} &= -\frac{7}{66}\frac{8f_{44}f_{4444}-13f_{444}^2}{f_{44}^2};\\
W_{66} &= -\frac{1}{198}\frac{252 f_{44}f_{4444}-437 f_{444}^2}{f_{44}^2},
\end{align*}
and we have the following syzygy:
\[
3W_{75}-2W_{66}+\frac{5}{648}I^2=0.
\]
Hence, the equality $W=0$ implies also that $I=0$. Again, we see that vanishing
of the first non-trivial generalized Wilczynski invariant implies that the
function $f$ is actually quadratic in the highest derivative. This completes
the case $n=4$ and the proof of the Lemma.
\end{proof}

\begin{rem} Note that for $n=2$ Theorem \ref{basiclem} does not hold. In particular, the Lagrangian
$L=(y'')^{1/3}\, dx$
has trivial Wilczynski invariants, but the associated Euler-Lagrange equation
$3y''y^{(4)}-5(y''')^2=0$ is not trivializable and has only 6-dimensional
symmetry algebra. On the other hand, Theorem \ref{Wilthm} is valid for $n=2$ and not only for distributions associated with second order Lagrangians but for any rank 2 distribution in $\mathbb R^5$ with the small growth vector $(2,3,5)$.
The rank 2 distributions $D$ corresponding
to the equations $z'=(y'')^{1/3}$ and $z'=(y'')^2$ are equivalent and
have the 14-dimensional symmetry algebra, which
is the maximal possible algebra for distributions under the consideration.
\end{rem}

\begin{rem}
Direct analysis of the symmetry classification of all ordinary differential
equations (see~\cite{olver}) shows that all Euler--Lagrange equations
of order $2n$ with the symmetry algebra of dimension at least $2n+1$
are exhausted (modulo contact transformations) by the following ones:

\begin{center}
\begin{tabular}{|c|c|c|}
\hline
Equation $\E$ & Lagrangian $L$ & $\dim\sym(\E)$ \\
\hline
$y_{2n}=0$ & $y_n^2\,dx$ & $2n+4$\\
\hline
$y^2_n+\sum_{i=0}^{n-1}c_iy_{2i}=0$ &
$\left(y_n^2+\sum_{i=0}^{n-1}c_iy_i^2\right)\,dx$ & $2n+2$\\
\hline
$9y_3^2y_6-45y_3y_4y_5+40y_4^3 = 0$ & $y_3^{1/3}\,dx$ & $7$\\
\hline
\end{tabular}
\end{center}

\smallskip
Note also that we always have $\dim\sym(L)=\dim\sym(\E)+1$.
\end{rem}

\section*{Acknowledgments} We would like to thank Andrei Agrachev, Ian Anderson, Mark Fels, and Eugene Ferapontov for valuable discussions on the subject of this paper.

\end{document}